\documentclass[12pt]{article}
\usepackage[all]{xy}
\usepackage{amssymb,amsmath,amsthm,amscd,amsfonts}
\newtheorem{theorem}{Theorem}[section]

\newtheorem{proposition}[theorem]{Proposition}

\newtheorem{definition}[theorem]{Definition}
\newtheorem{example}[theorem]{Example}

\newtheorem{remark}[theorem]{Remark}
\numberwithin{equation}{section}
\newcommand{\arrow}{{\longrightarrow}}
\newcommand\ib{\mathbb{Z}}
\newcommand{\ci}{\mathcal{I}}
\newcommand{\sbs}{\subseteq}
\def\im{{\rm im}}
\def\en{{\rm End}}
\def\id{{\rm id}}

\newcommand{\la}{\lambda}

\date{}
\title{\bf Strongly Hopfian and Co-Hopfian Acts over Monoids: Structure and Characterizations}
\author{{\bf Ali Madanshekaf} \and {\bf Farideh Farsad} \\
Department of Mathematics\\Faculty of Mathematics, Statistics and Computer Science\\
Semnan University\\ P. O. Box 35131-19111\\
Semnan\\
Iran\\ email: amadanshekaf@semnan.ac.ir\\
}
\date{}
\begin{document}
\maketitle
\begin{abstract}
In this paper, we introduce and explore new classes of $S$-acts over a monoid $S$, namely, strongly Hopfian and strongly co-Hopfian acts, as well as their weaker counterparts, Hopfian and co-Hopfian acts. We investigate the relationships between these newly defined structures and well-studied classes of $S$-acts, including Noetherian, Artinian, injective, projective, quasi-injective, and quasi-projective acts. A key result shows that, under certain conditions, a quasi-projective (respectively quasi-injective) $S$-act that is strongly co-Hopfian (respectively strongly Hopfian) is also strongly Hopfian (respectively strongly co-Hopfian). Moreover, we provide a variety of examples and structural results concerning the behavior of subacts and quotient acts of strongly Hopfian and strongly co-Hopfian $S$-acts, further elucidating the internal structure and interrelationships within this extended framework.
\end{abstract}
%-----------------------------------------------------------------------
AMS {\it subject classification}: 20M30, 20M50. 
\section{Introduction and Preliminaries}
%%%%%%%%%%%%%%%%%%%%%%%%%%%%%%%%%%%%%%%%%%%%%%%%%%%%%%%%%%%%%%
The study of modules (rings) by properties of their endomorphisms is a classical research subject. In 1986, Hiremath~\cite{Hir} introduced the concepts of Hopfian modules and rings. Later, the dual concepts co-Hopfian modules and rings were given. An $R$-module  is said to be Hopfian (respectively co-Hopfian) if any surjective (respectively injective) $R$-endomomorphism is automatically an isomorphism. The terms ``Hopfian" and ``co-Hopfian" have arisen since the 1960s, and are said to be in honor of Heinz Hopf, who asked if all finitely generated groups were Hopfian and his use of the concept of the Hopfian group in his work on fundamental groups of
surfaces. Hopfian and co-Hopfian modules (rings) have been investigated by several authors. \\
Likewise, the study of $S$-acts by properties of their endomorphisms is a classical research subject (see for example~\cite{B, FK, KM, KN}). For instance, the well-known result that every surjective (injective) endomorphism in a Noetherian (Artinian) $S$-act is an automorphism, motivates and leads partly, the introduction and investigation of Hopfian, co-Hopfian $S$-acts (see also~\cite{AM,KK}).

In view of the connection between $S$-acts and modules, in this paper, we introduce strongly Hopfian
(respectively strongly co-Hopfian) $S$-acts and study some properties of them using the techniques similar to those used in modules. More precisely,  we start in section 2 from the following facts: the proof that every Noetherian (respectively Artinian) $S$-act $A$, is Hopfian (respectively co-Hopfian), involves only stationary chains of the type $\mathcal{K}_f\subseteq\mathcal{K}_{f^2}\subseteq\cdots$ (respectively
$\mathcal{I}_{f}\supseteq\mathcal{I}_{f^2}\supseteq\cdots )$ where $f$ is any endomorphism of $A$. There exist $S$-acts that satisfy these conditions but they are neither Artinian nor Noetherian. Furthermore, there exist
Hopfian and co-Hopfian $S$-acts that do not satisfy the previous conditions. In section 3, we introduce the class of strongly Hopfian (respectively strongly co-Hopfian) $S$-acts, as the $S$-act $A$ satisfying that the chain
$\mathcal{K}_f\subseteq \mathcal{K}_{f^2}\subseteq\cdots$ (respectively $\mathcal{I}_{f}\supseteq\mathcal{I}_{f^2}\supseteq\cdots)$ is stationary for every $f\in \en_S (A)$. It is immediately clear that  strongly Hopfian implies Hopfian and strongly co-Hopfian implies co-Hopfian. In  Example \ref{not st. Hopfian} below we will present some Hopfian $S$-act that are not strongly Hopfian. It is shown that if $B$ is fully invariant subact of an $S$-act $A$ and $B, A/B$ are strongly Hopfian then so is $A$ (see Theorem~\ref{fully invariant}). 
In the last two sections we pay attention to interaction between strongly Hopfian and strongly co-Hopfian objects with quasi-projective and quasi-injective objects (see Theorems~\ref{quasi-in+st Hopf} and \ref{quasi-proj+st co-Hopf}). Also, some examples and results that are connected with the subacts and quotients of strongly Hopfian and strongly co-Hopfian $S$-acts are presented. 
%\textcolor{blue}{
	A preliminary version of this work was presented at [The 1st international conference of educational management of Iran, 2017] by the second author. In this paper, we provide a full and extended version with additional results, detailed proofs, and broader discussion.
	%}

Now we give some notation and terminology. Throughout this paper, $S$ always stands for a monoid, and $\mathbb{N}$ for the set of natural numbers. 
%Let $S$ be a monoid with identity $1$. 
A non-empty set $A$ is called a right {\it $S$-act} if there is a map $\mu: A\times S\to A$, called its action and denoted by $(a, s)\mapsto as$, satisfying  $a1 = a$ and $a(st)=(as)t$, for all $a\in A$, and $s, t\in S$. The notion of a left $S$-act is defined dually. To simplify, by an $S$-act we mean a right $S$-act. The category of all $S$-acts, with $S$-homomorphisms  (or $S$-act maps) (i.e., $f : A\to B$ with $f(as) = f(a)s$, for $s\in S, a\in A$)
is denoted by \textbf{Act}-$S$. Clearly $S$ itself is an $S$-act with its operation as the action. Let $B\subseteq A$ be a non-empty subset of $A$. Then $B$ is called a subact of $A$ if $bs\in B$ for all $s\in S$ and $b\in B$. 

A  congruence on an $S$-act $A$ is an equivalence relation $\theta$ with the property
that if $a\mathbin{\theta} a^{\prime }$ for $a, a^{\prime }\in A,$ then $
as\mathbin{\theta} a^{\prime }s$, for all $s\in S$. Then the set of equivalence classes 
$A/\theta$ can be made into an $S$-act in such a way that the natural map
\begin{displaymath}
	\pi_{\theta}: A \longrightarrow A/\theta; \quad a \mapsto [a]_{\theta}
\end{displaymath} 
is an $S$-act map. It is called the canonical epimorphism.  As two examples of the situation that we encounter with congruence,  we have the following:\\
Let $A$ be a right $S$-act.

(1) Any subact $B\subseteq A$ defines the Rees congruence $\rho_B$ on $A$, by setting $a\mathbin{\rho_{B}  } a^{\prime }$ iff $a, a^{\prime }\in
B$ or $a= a^{\prime }$. We denote the resulting factor act by $A/B$ and call it the Rees factor act of $A$ by the subact $B$.

(2) Let $f$ be an endomorphism of $A$. We define two relations on $A$ by
$$\mathcal{K}_f=\{(a_1, a_2)\in A\times A \mid f(a_1)=f(a_2)\},\quad\mathcal{I}_{f}=(\im f\times \im f)\cup \Delta_{A}$$
where $\Delta_{A}$ is the diagonal equivalence relation on $A$ and $\im f$ is the image of the map $f$. It is clear
that both $\mathcal{K}_f$ and $\ci_{f}$ are congruences on $A$. 
For basic definitions and terminology relating semigroups and acts over monoids, we refer the reader to~\cite{H,KKM,MT}. 
%\marginpar{refer\cite{M} in another place}
%%%%%%%%%%%%%%%%%%%%%%%%%%%%%%%%%%%%%%%%%%%%%%%%%%%%%%%%%%%%
\section{Hopfian and co-Hopfian $S$-acts}
%%%%%%%%%%%%%%%%%%%%%%%%%%%%%%%%%%%%%%%%%%%%%%%%%%%%%%%%%%%%%%
We recall the following definitions and facts which we deal with in this paper.
\begin{definition}{\rm 
		We call an $S$-act  {\it Noetherian} when its congruences satisfy the ascending chain condition. Also, we call an $S$-act  {\it Artinian} when its congruences 	satisfy the descending chain condition.}
\end{definition}
We refer the reader to~\cite{NP,G1,G2}, for basic results and terminology about the above notions.
\begin{definition}
	{\rm  (See~\cite{AM}
		% and~\cite{KK}
		 also.)\\ An $S$-act $A$ is called {\it Hopfian} if every surjective endomorphism
		of $A$ is an automorphism, or equivalently, it is not isomorphic to any proper quotient of itself. An $S$-act $A$ is called {\it co-Hopfian} if every injective endomorphism of $A$ is an automorphism, or, equivalently, if it is
		not isomorphic to any proper subact of itself.}
\end{definition}
\begin{example}
	\begin{enumerate}{\rm
			\item[(a)] Every finite $S$-act is Hopfian and co-Hopfian.
			\item[(b)] Let $A$ be an arbitrary $S$-act. Then infinite product $P = A\times A\times\cdots$ is not Hopfian
			because the surjective $S$-map 	$\varphi: P\to P$ given by
			$$(a_1,a_2,\cdots )\mapsto (a_2,a_3,\cdots )$$ 
			that is not injective so that it is not an automorphism.
			\item[(c)] Recall that an algebra  is \emph{simple} if it has no non-trivial congeuences (e.g., see \cite{BS}). Then 	every simple $S$-act $A$ with $|A|>1$  is Hopfian. In fact, let $f$ be a surjective endomorphism of $A$. If $\mathcal{K}_f = \Delta_A$, then we have done. Otherwise, $\mathcal{K}_f = A\times A$ that shows $f$ is a fixed map which is a contradiction with cardinality of $A$. 
		}
	\end{enumerate}
\end{example}
\begin{proposition}
	Every Noetherian $S$-act is Hopfian.
\end{proposition}
\begin{proof}
	Let $A$ be a Noetherian $S$-act. For every surjective endomorphism
	$f$ of $A$, the ascending chain $\mathcal{K}_f\subseteq\mathcal{K}_{f^{2}}\subseteq\mathcal{K}_{f^{3}}\subseteq\cdots$ stabilizes, so there exists $n\in \mathbb N$ such that $\mathcal{K}_{f^{n}}=\mathcal{K}_{f^{n+1}}=\cdots$. We show that 	$\mathcal{K}_f=\Delta_{A}$.	For $(a,b)\in\mathcal{K}_f\subseteq A\times A$,	
	 since $f^n$ is surjective, there exists $(c , d)\in A\times A$ such that $f^n(c)=a$ and $f^n(d)=b$. In view of $f(a)=f(b)$, we get $(c , d)\in\mathcal{K}_{f^{n+1}}= \mathcal{K}_{f^{n}}$  so that $a=f^{n}(c)=f^{n}(d)=b.$ Hence $\mathcal{K}_f=\Delta_{A}$ which implies that $f$ is injective, as required.
\end{proof}
\begin{proposition}
	Every Artinian $S$-act is co-Hopfian.
\end{proposition}
\begin{proof}
	Let $A$ be an Artinian $S$-act. For every injective endomorphism
	$f$ of $A$, the descending chain $\mathcal{I}_{f}\supseteq\mathcal{I}_{f^2}\supseteq\cdots$ stabilizes.
	So there exists $n\in \mathbb{N}$
	such that $\mathcal{I}_{f^n}=\mathcal{I}_{f^{n+1}}=\cdots $. Now for every $a\in A$, consider
	any $b\in A$ with $b\neq a$, we have $(f^{n}(a), f^{n}(b))\in
	\mathcal{I}_{f^n}=\mathcal{I}_{f^{n+1}}$, but $f^n$ is injective, this gives that
	$f^{n}(a)\neq f^{n}(b)$. Hence $(f^{n}(a),f^{n}(b))\in
	\im f^{n+1}\times \im f^{n+1}$, this means that there exists
	$(x,y)\in A\times A$ such that $f^{n+1}(x)=f^{n}(a)$ and
	$f^{n+1}(y)=f^{n}(b)$. Again by injectivity of $f^n$ we conclude
	that $f(x)=a$ and $f(y)=b$. Consequently, $a\in \im f$ which implies $f$ is surjective, as we needed.
\end{proof}
One might ask whether the converse of the above propositions are true. By providing the following example we answer the above question.
\begin{example}\label{Q}
	%\marginpar{This  example should be repaired using the thesis}
	{\rm Consider the set of rational numbers $\mathbb{Q}$  as a right $(\mathbb{Z},\cdot )$-act with multiplication by integers as operation (see also~\cite[page 64]{KKM}). Then $\mathbb{Q}$ is a non-Noetherian and non-Artinian $\mathbb Z$-act which is Hopfian and co-Hopfian. To see this, if $f\in\en_\mathbb{Z} (\mathbb{Q})$, then $f(q) = qf(1)$ for any rational $q$. So all nonzero endomorphisms are automorphisms. This	gives that $\mathbb{Q}$ is both Hopfian and co-Hopfian $S$-act. The non-Noetherianity follows from $\rho_1 \subsetneq \rho_2 \subsetneq \rho_3 \subsetneq \cdots$ where $\rho_j$ is the Rees congruence given by $(1/2^j)\mathbb{Z}$, where, $j\in \mathbb{N}$ (see \cite{G1,G2} also).
					
		On the other hand, for every subact $B$ of $\mathbb{Q}$, consider the Rees congruence $\rho_B$ on $\mathbb{Q}$ and using the  fact that for two subacts $M$ and $N$ of any act with $M\subseteq N$ 	one has $\rho_M\subseteq \rho_N$. Then according to this fact we have the 	descending chain $\rho_{\mathbb {Z}}\supsetneq	\rho_{{2\mathbb{Z}}}\supsetneq \rho_{4\mathbb{Z}}\supsetneq\cdots$ 	that does not stabilize. Therefore $\mathbb{Q}_{\mathbb{Z}}$ is not Artinian.}
\end{example}
%%%%%%%%%%%%%%%%%%%%%%%%%%%%%%%%%%%%%%%%%%%%%%%%%%%%%%%%%%%%%%%%%%%%%%%%%%%%%%%%%%%%%%%%%%%%%%%%%%%
\section{Strongly Hopfian and Strongly co-Hopfian $S$-acts}
%%%%%%%%%%%%%%%%%%%%%%%%%%%%%%%%%%%%%%%%%%%%%%%%%%%%%%%%%%%%%%%%%%%%%%%%%%%%%%%%%%%%%%%%%%%%%%%%%%%
In realizing that every Noetherian (respectively Artinian) $S$-act $A$ is Hopfian (respectively co-Hopfian), the ACC (respectively DCC) is only applied to chain of congruences of the form
\[
(a)\quad\mathcal{K}_f\subseteq\mathcal{K}_{f^{2}}\subseteq\mathcal{K}_{f^{3}}\subseteq\cdots \quad (\textrm{respectively,} \  
(b)\quad\mathcal{I}_{f}\supseteq\mathcal{I}_{f^2}\supseteq\cdots )
\]
where $f$ is any endomorphism of $A$. As we have already seen, we give the following example:\\
Since the $\mathbb{Z}$-endomorphisms of $\mathbb{Q}$ are either zero or automorphisms, it follow that $\mathbb{Q}$ is a
non-Noetherian and non-Artinian $\mathbb{Z}$-act which satisfy both conditions $(a)$ and $(b)$.\\
The above remark motivates  the following definition.
\begin{definition}
	{\rm An $S$-act $A$ is called {\it strongly Hopfian} if for every endomorphism $f$ of $A$ the ascending chain $\mathcal{K}_f\subseteq
		\mathcal{K}_{f^2}\subseteq \mathcal{K}_{f^3}\subseteq\cdots $ stabilizes. That is, there exist a positive integer $n$ such that $\mathcal{K}_{f^m}=  \mathcal{K}_{f^n}$ for all $m\geq n.$ An $S$-act $A$ is called {\it strongly co-Hopfian} if for every
		endomorphism $f$ of $A$ the descending chain $\mathcal{I}_{f}\supseteq\mathcal{I}_{f^2}\supseteq\mathcal{I}_{f^3}\cdots$ stabilizes. That is, there exist a positive integer $n$ such that $\ci_{f^m}=  \ci_{f^n}$ for all $m\geq n.$}
\end{definition}
	It is immediately clear that  strongly Hopfian implies Hopfian and strongly co-Hopfian implies co-Hopfian. In  Example \ref{not st. Hopfian} below we will present some Hopfian $S$-acts that are not strongly Hopfian.	\\
Now we obtain:
\begin{proposition}\label{strongly Hopfian}
	For an $S$-act $A$ the following statements are equivalent:\\
	$(1)~A$ is strongly Hopfian.\\
	$(2)$ For every $f$ in $\en_{S}(A)$ there exists an integer $n\geq 1$ such that $\mathcal{K}_{f^n}=\mathcal{K}_{f^{n+1}}$.\\
	$(3)$ For every $f$ in $\en_{S}(A)$ there exist an integer $n\geq 1$ such that $\mathcal{I}_{f^n}\bigcap\mathcal{K}_{f^n}=\Delta_A$.
\end{proposition}
\begin{proof}
	(1)$\Rightarrow$(2) This is always true by definition of strongly Hopfian $S$-acts.\\
	(2)$\Rightarrow$(1) Suppose that $f\in\en_{S}(A)$. If, for some $n$,
		$\im(f^n)$ consists of one point, we are done. In the complementary case,  we have $\im(f^n) = \im(f^{n+1})$ for some positive integer $n$. Apply $f$ to get the equality $\im(f^{n+1}) = \im(f^{n+2})$ and so on.
		% For every $f\in\en_{S}(A)$ satisfying in 	the assumption, the proof is easy by induction.
	Therefore, statements
	(1) and (2) are equivalent.\\
	To complete the proof it is enough to show that the remaining implications.\\
	(2) $\Rightarrow$ (3) 
	Suppose that $f\in\en_{S}(A)$ and choosing an integer $n\geq 1$ such that $\mathcal{K}_{f^n}=\mathcal{K}_{f^{n+1}}$.
	Let $(a,b)\in\mathcal{I}_{f^n}\bigcap\mathcal{K}_{f^n}.$ Then by $(a,b)\in \mathcal{I}_{f^n}$ we have either $(a,b)\in
	\Delta_A$ or $(a,b)\in \im f^{n}\times \im f^{n}$, it follow that there exists $(x,y)\in A\times A$ such that $a=f^{n}(x)$ and
	$b=f^{n}(y)$. Now by $(a,b)\in\mathcal{K}_{f^n}$ we get $f^n(a)=f^{2n}(x)=f^{2n}(y)=f^n(b)$. By assumption $\mathcal{K}_{f^{2n}}=\mathcal{K}_{f^{n}}$ 	and so $a=f^{n}(x)=f^{n}(y)=b$, that is, $(a,b)\in \Delta_A$.\\
	(3) $\Rightarrow$ (2): Suppose that $f\in\en_{S}(A)$ and choosing an integer $n\geq 1$  that satisfies (3).
		For every $(a,b)\in\mathcal{K}_{f^{n+1}}$, we have $(f^{n}(a),f^{n}(b))\in\mathcal{K}_f\subseteq\mathcal{K}_{f^{n}}$. Therefore $((f^{n}(a),f^{n}(b))\in \mathcal{I}_{f^n}\bigcap\mathcal{K}_{f^n}=\Delta_A$ and
	so $(a,b)\in\mathcal{K}_{f^{n}}$. This shows that $\mathcal{K}_{f^{n+1}}\sbs \mathcal{K}_{f^n}$. Hence, 
	%$A$ is strongly Hopfian.
			(2) follows.
	\end{proof}
In the following result, we present a slight extension of Proposition 3.3 in \cite{FS}.
\begin{proposition}\label{strongly co-Hopfian}
	For an $S$-act $A$ the following statements are equivalent:\\
	$(1)~A$ is strongly co-Hopfian.\\
	$(2)$ For every $f$ in $\en_{S}(A)$ there exists an integer $n\geq 1$ such that $\mathcal{I}_{f^n}=\mathcal{I}_{f^{n+1}}$.\\
	$(3)$  For every $f$ in $\en_{S}(A)$ there exists an integer $n\geq 1$ such that $\mathcal{I}_{f^n}\vee\mathcal{K}_{f^n}=A\times A$.
\end{proposition}
\begin{proof}
	(1)$\Rightarrow$(2) The assertion is obvious.\\ 
	(2)$\Rightarrow$(1) Given $f\in \en_{S}(A)$  such that $\mathcal{I}_{f^n}=\mathcal{I}_{f^{n+1}}$ for some $n\in \mathbb N$, one has $\im f^{n}=\im f^{n+1}$. We claim that  
	$\mathcal{I}_{f^n}=\mathcal{I}_{f^m}$ for every $m\geq n.$  It is enough to show that
	$\im f^{n}=\im f^{n+2}$. Firstly, we have $\im f^{n+2}\subseteq
	\im f^{n+1}=\im f^{n}$. Secondly, consider $x\in \im f^{n} = \im f^{n+1}$ then  there
	exists $a\in A$ such that $f(f^{n}(a))=x$. Now, $f^{n}(a)\in
	\im f^{n}=\im f^{n+1}$ and so there exists $b\in A$ such that
	$f^n(a) = f^{n+1}(b)$. This means $x\in \im f^{n+2}$, as we required. We
	conclude for every $m\in \mathbb{N}$ that $m\geq n$,
	$\im f^{n}=\im f^m$,   so that $\mathcal{I}_{f^{n}}=\mathcal{I}_{f^{m}}$, i.e., $A$ is
	strongly co-Hopfian.\\
	(3)$\Rightarrow$(2) Let $n$ be a positive integer and $\mathcal{I}_{f^n}\vee\mathcal{K}_{f^n}=A\times A.$ We will show that $\mathcal{I}_{f^n}=\mathcal{I}_{f^{n+1}}$. Let  $(x,y)\in \mathcal{I}_{f^n}$. If $(x,y)\in \Delta_A$ then
	$(x,y)\in \mathcal{I}_{f^{n+1}}$. Otherwise, $(x,y)\in\im f^{n}\times \im f^{n}$ then one has $x = f^n(u)$ and $y = f^n(v)$ 
	for some $u$ and $v$ in $A$. By the assumption $(u , v)\in\mathcal{I}_{f^n}\vee\mathcal{K}_{f^n}$, therefore by \cite[Proposition I.4.3]{KKM} there exist elements $b_1, b_2, \cdots , b_{n-1}\in A$ such that $(u , b_1)\in \tau_1, (b_1, b_2)\in \tau_2, \cdots , (b_{n-1} , v)\in \tau_n$ where $\tau_i\in \{\mathcal{I}_{f^n} , \mathcal{K}_{f^n} \}$ for $i=1, \cdots , n.$ 
	Without loss of generality we may assume that $n=2$ and there exist $b_1, b_2\in A$ for which one has $(u , b_1)\in\mathcal{I}_{f^n}, (b_1, b_2)\in\mathcal{K}_{f^n}, (b_2 , v)\in \mathcal{I}_{f^n}$. Then we
	distinguish four cases:
	\begin{enumerate}
		\item[(I)] $u = b_1, f^n( b_1) = f^n( b_2), b_2 = v;$
		\item[(II)] $u = b_1, f^n( b_1) = f^n( b_2), (b_2 , v)\in \im f^{n}\times \im f^{n};$
		\item[(III)] $(u , b_1)\in \im f^{n}\times \im f^{n}, f^n( b_1) = f^n( b_2), b_2 = v;$
		\item[(IV)] $(u , b_1)\in \im f^{n}\times \im f^{n}, f^n( b_1) = f^n( b_2), (b_2 , v)\in \im f^{n}\times \im f^{n}.$
	\end{enumerate}
	As you can see the discussion below in all cases it implies $(x,y)\in \mathcal{I}_{f^{n+1}}$. \\
	Case (I). We have $x = f^n(u) = f^n( b_1)=  f^n( b_2)=  f^n(v) = y.$ Therefore $(x,y)\in \Delta_A$. \\
	Case (II). We have $x = f^n(u) = f^n( b_1)=  f^n( b_2)$ and $b_2 =f^n(r), v =f^n(s)$ for some  $r, s\in A$. One has $(x,y) = (f^{2n}(r), f^{2n}(s))\in \im f^{n+1}\times \im f^{n+1}$.\\
	Case (III). It is similar to  Case (II).\\
	Case (IV). We have $ u = f^n(r), v= f^n(s)$ where $r, s\in A$. Then 
	\[(x,y) = (f^{2n}(r), f^{2n}(s))\in \im f^{n+1}\times \im f^{n+1}. \] 
	Thus $\mathcal{I}_{f^n}\subseteq\mathcal{I}_{f^{n+1}}$. The converse inclusion is clear.\\
	(2)$\Rightarrow$(3) 
	Let $n$ be a positive integer and $\mathcal{I}_{f^n}=\mathcal{I}_{f^{n+1}}.$ We will show that $\mathcal{I}_{f^n}\vee\mathcal{K}_{f^n}=A\times A$. Let  $(a , b)\in A\times A$. Then $(f^n(a) , f^n(b))\in\mathcal{I}_{f^n} .$ If $f^n(a) = f^n(b)$ then we have  $(a , b)\in \mathcal{I}_{f^n}\vee\mathcal{K}_{f^n}$ and we are done.	
	If $f^n(a)\ne f^n(b)$ then, by the assumption  $\mathcal{I}_{f^n}=\mathcal{I}_{f^{n+1}},$ we get $f^n(a)= f^{2n}(c)$ and $f^n(b) = f^{2n}(d)$ for some $c, d\in A$. We then  have $(a , f^n(c))\in\mathcal{K}_{f^n}, (f^n(c) , f^n(d))\in\mathcal{I}_{f^n}, (f^n(d) , b)\in \mathcal{K}_{f^n}$ so that $(a , b)\in\mathcal{I}_{f^n}\vee\mathcal{K}_{f^n}$  by \cite[Proposition I.4.3]{KKM}. Therefore, $\mathcal{I}_{f^n}\vee\mathcal{K}_{f^n}=A\times A$.
\end{proof}
% But we have two  open questions:\\
%\marginpar{\fbox{remain!}} 
%{\bf Questions}:\\

%1) Is any Hopfian $S$-act $A$ strongly Hopfian?\\
%2) Is any co-Hopfian $S$-act $A$ strongly co-Hopfian?
%Given a monoid $S$ and an element $s\in S$, we define $\rho_s : S\arrow S$ by $\rho_s (x) := xs$, and we put
%$l(s)= \mathcal{K}_{\rho_s} = \{ (x , y)\in S\times S : xs = ys\}.$

Given a monoid $S$ and an element $s\in S$, we define $S$-homomorphism $\la_s : S\arrow S$ by $\la_s (x) := sx$, and  put
\begin{displaymath}
	r(s)= \mathcal{K}_{\la_s} = \{ (x , y)\in S\times S : sx = sy\}.
\end{displaymath} 
Then we have
\begin{proposition}\label{st. Hopfianity of $S_S}
	For a monoid $S$, one has:\\
	$(1)$ $S_S$ is  strongly Hopfian if and only if for every $s\in S$ there exists $n\in \mathbb{N}$ such that $r(s^n) = r(s^{n+1}).$\\
	$(2)$ $S_S$ is  strongly co-Hopfian if and only if for every $s\in S$ there exists $n\in \mathbb{N}$ and $t\in S$ such that $s^n = s^{n+1}t$.
\end{proposition}
\begin{proof}
	$(1)$ First note that any $S$-homomorphism $f: S_S\to S_S$ is completely determined by $f(1)$. In fact, let us take $f(1) = a\in S$, then for any $s\in S$ we have $f(s) = f(1)s = as = \la_a(s)$ so that $f= \la_a$. Therefore, 
	$$\en_S(S_S) =\{\la_a\mid a\in S\}.$$ We now have:\\
	$S_S$ is  a  strongly  Hopfian  act if and only if for all $f\in \en_S(S_S) =\{\la_a\mid a\in S\},$ there  exists $n\in\mathbb{N}$ such that $\mathcal{K}_{f^n}=\mathcal{K}_{f^{n+1}}$ if and only if  there  exists $n\in\mathbb{N}$ such that $\mathcal{K}_{\la^{n}_a}=\mathcal{K}_{\la^{n+1}_a}$  if and only if there exists $n\in\mathbb{N}$ such that $\mathcal{K}_{\la_{a^n}}=\mathcal{K}_{\la_{a^{n+1}}}$  if and only if for every $s\in S$ there exists $n\in \mathbb{N}$ such that $r(s^n) = r(s^{n+1}).$ \\
	$(2)$ Similar to (1),  $S_S$ is  a  strongly  co-Hopfian  act if and only if for all $f\in \en_S(S_S) =\{\la_a\mid a\in S\},$ there  exists $n\in\mathbb{N}$ such that $\ci_{f^n} =\ci_{f^{n+1}}$ if and only if for all $f\in \en_S(S_S) =\{\la_a\mid a\in S\},$ there  exists $n\in\mathbb{N}$ such that $\im {f^n} =\im {f^{n+1}}$ if and only if  for all $a\in S$
	 there  exists $n\in\mathbb{N}$ such that $\im \la_{a^n}  =\im \la_{a^{n+1}}$  if and only if  	for all $a\in S$ there  exists $n\in\mathbb{N}$ such that  $a^nS  =a^{n+1}S$  if and only if 
			for all $a\in S$ there  exists $n\in\mathbb{N}$ and  $b\in S$ such that $a^n  =a^{n+1}b.$
\end{proof}	
%\textcolor{blue}{
	So far, it has not been established whether the class of strongly Hopfian $S$-acts forms a proper subclass of Hopfian 	$S$-acts. In the following example, we demonstrate that it indeed is a proper subclass.
	%}
	\begin{example}\label{not st. Hopfian}
{\rm		To give an example of an  $S$-act which is Hopfian but not strongly Hopfian, first, notice that it is easy to see that any commutative monoid as an $S$-act over itself is Hopfian. In fact, if  $S$ is such a monoid and let $f$ be a surjective endomorphism of $S_S$, then there exist an element $s\in S$ such that $f(1)s = f(s) =1.$ If $x, y\in S$ such that $f(x) = f(y)$, then $f(1) x = f(1)y$. By multiplying $s$ to this equation on the right we get $x = y$, because $S$ is commutative. Taking into account this fact, for any prime number $p$, let us now put $S =\prod_{n\geq 1}^{}\ib/p^n\ib$ with the usual multipication. Let $x = (x_n)$ with $x_n = p + p^n\ib$. For every $k\geq 1$, let $y = (y_n$) where  $y_n = 0+ p^n\ib$ if $n=k + 1$ and $y_{k+1} = 1+ p^{k+1}\ib$, then we have $(y, 0)\in r(x^{k+1})$ but $(y, 0)\not\in r(x^{k})$. Thus by proposition \ref{st. Hopfianity of $S_S}, the $S_S$ is not strongly Hopfian act.}
	\end{example}	
\begin{remark}
	{\rm Subacts  of strongly Hopfian (respectively strongly co-Hopfian) $S$-acts need not to be
		strongly Hopfian (respectively strongly co-Hopfian) as the following example show:}
\end{remark}
\begin{example}
	$\mathbb{Q}$ as an $\mathbb{Z}$-act is strongly co-Hopfian
	but its subact $\mathbb{Z}$ is not even co-Hopfian. To see that,
	take the map $f: \mathbb{Z}\rightarrow \mathbb{Z}$ given by $f(n)=2n$,
	for every $n\in \mathbb{Z}$. This homomorphism is injective but is not surjective.
\end{example}

\begin{theorem}
	Every proper retract of a strongly Hopfian $S$-act is strongly Hopfian.
\end{theorem}
\begin{proof}
	Assume that  $A$ is a proper retract of a strongly Hopfian $S$-act  $B,$  i.e., we have
	$S$-homomorphisms $\gamma : A\rightarrow B$ and $\pi : B \rightarrow A$ such that $\pi\gamma =\id_A$ and $\gamma$ non-bijective.
	%	$A\neq B.$  
	Let $f$ be an endomorphism of $A$. In the following diagram 
	\[
	\xymatrix{A\ar@<0.5ex>[r]^{\gamma}\ar[d]_f & B\ar@<0.5ex>[l]^{\pi}\ar[d]^g\\
		A\ar@<0.5ex>[r]^{\gamma} & B\ar@<0.5ex>[l]^{\pi}}
	\]
	define $g : B \rightarrow B$ by the rule $g:=\gamma f\pi$, which  is an
	endomorphism of strongly Hopfian $S$-act $B$ and so there exists
	$n\in \mathbb{N}$ such that $\mathcal{K}_{g^{n}}=\mathcal{K}_{g^{n+1}}$. We show
	that $\mathcal{K}_{f^n}=\mathcal{K}_{f^{n+1}}$. It is sufficient to show that $\mathcal{K}_{f^{n+1}}\sbs\mathcal{K}_{f^n}$. 

		Let $(x,y)\in\mathcal{K}_{f^{n+1}}$, so $f^{n+1}(x) = f^{n+1}(y).$ As $\pi : B \to A$ is a surjective 
	homomorphism (because of the existence of right inverse),
	there exists $(a,b)\in B\times B$ such that $\pi(a) = x, \pi (b) = y$ and so
	\[\begin{array}{rcl}
		f^{n+1}(\pi (a))=f^{n+1}(\pi (b))  & \Rightarrow &\gamma f^{n+1}(\pi (a))=\gamma f^{n+1}(\pi (b))\\
		& \Rightarrow & g^{n+1}(a)=g^{n+1}(b).
	\end{array}\]
	Hence $(a,b)\in\mathcal{K}_{g^{n+1}}=\mathcal{K}_{g^n}$ therefore $\gamma f^{n}(\pi (a))=\gamma f^{n}(\pi
	(b))$. Since $\gamma$ is injective we obtain $f^{n}(x)=f^{n}(y)$ and so
	$(x,y)\in\mathcal{K}_{f^n}$. Thus $\mathcal{K}_{f^{n+1}}=\mathcal{K}_{f^n}$ and  $A$ is strongly Hopfian.
\end{proof}

\begin{theorem}
		Let  $A$ be a strongly co-Hopfian act and $h:A\to B$ a surjective morphism such that every endomorphism of $B$ is induced by an endomorphism of $A$
		(e.g. $h$ has a section). Then $B$ is strongly co-Hopfian.
	\end{theorem}
	\begin{proof}
		Let $f$ be an endomorphism of $B$. By the hypothesis, there exists an endomorphism $g$ of $A$ such that the diagram 
		$$
		\xymatrix{A_S\ar[r]^{h}\ar[d]_{g} &B_S\ar[d]^{f}\\A_S\ar[r]_{h}&B_S}
		$$
		commutes. Likewise, we have $\mathcal{I}_{g^n}=\mathcal{I}_{g^{n+1}}$ for some $n\in \mathbb N$. We will show that $\mathcal{I}_{f^n}=\mathcal{I}_{f^{n+1}}$. 
		Let  $(b_1,b_2)\in \mathcal{I}_{f^n}$ and $b_1\ne b_2$. Then $(b_1,b_2)\in\im f^n\times \im f^n$. There exist $u_1, u_2\in B$ such that $b_1 = f^n(u_1)$ and $b_2 = f^n(u_2)$. Also, there exist $a_1, a_2\in A$ such that $u_1 = h(a_1)$ and $u_2 = h(a_2)$.   As, $fh = hg$ we have $f^mh = hg^m$ for all $m\geq 1$. Then, $b_1 = hg^n(a_1)$ and $b_2 = hg^n(a_2)$. As, $\mathcal{I}_{g^n}=\mathcal{I}_{g^{n+1}}$ we have $g^n(a_1) = g^n(a_2)$ or there exist $x, y\in A$ such that $g^n(a_1) = g^{n+1}(x)$ and $g^n(a_2) = g^{n+1}(y)$. In the former case we get $b_1 = b_2$ and in the latter case we have $b_1 = hg^n(a_1) = hg^{n+1}(x) = f^{n+1}h(x)$ and $b_2 = hg^n(a_2) = hg^{n+1}(y) = f^{n+1}h(y)$. Then, $(b_1,b_2)\in 
		\im f^{n+1}\times \im f^{n+1}$. Therefore, $\mathcal{I}_{f^n}\subseteq\mathcal{I}_{f^{n+1}}$ and we have done.
	\end{proof}

\begin{theorem}\label{fully invariant}
	If $B$ is fully invariant subact of an $S$-act $A$ and $B, A/B$ are strongly Hopfian then so is $A$.
\end{theorem}
\begin{proof}
	Recall that $\rho_{B}$ is the Rees congruence on $A$ associated to the subact $B$ of $A$. Let $f$ be an endomorphism of $A$. Consider the induced homomorphism
	$\bar{f}: A/\rho_{B}\rightarrow A/\rho_{B}$ given by
	$\bar{f}([a]_{\rho_{B}})=[f(a)]_{\rho_{B}}$ which is well defined because  $B$ being fully invariant. As $\bar{f}\in \en_S (A/B)$ 
	there exists $n\in \mathbb{N}$ such that
	$\mathcal{K}_{\bar{f}^n}=\mathcal{K}_{\bar{f}^{n+1}}$. Define $g:=f|_B$ which  is an endomorphism of $B,$ thus there exists
	$m\in \mathbb{N}$ such that $\mathcal{K}_{g^m}=\mathcal{K}_{g^{m+1}}$. Let
	$k=\max\{m,n\}$, we show that $\mathcal{K}_{f^{2k}}=\mathcal{K}_{f^{2k+1}}$. Given
	$(x,y)\in\mathcal{K}_{f^{2k+1}}$ so $f^{2k+1}(x)=f^{2k+1}(y)$. This implies that $([x]_{\rho_{B}},[y]_{\rho_{B}})\in
	\mathcal{K}_{\bar{f}^{2k+1}}=\mathcal{K}_{\bar{f}^{k}}$ so 	$\bar{f}^{k}([x]_{\rho_{B}})=\bar{f}^{k}([y]_{\rho_{B}})$
	or ${[f^{k}(x)]}_{\rho_{B}}={[f^{k}(x)]}_{\rho_{B}}.$ We distinguish 	two cases. First, $f^{k}(x)=f^{k}(y)$ then  $(x,y)\in \mathcal{K}_{f^k}$ and since   $\mathcal{K}_{f^k}\subseteq \mathcal{K}_{f^{2k}}$ hence $(x,y)\in \mathcal{K}_{f^{2k}}$. Second, $\{f^{k}(x),f^{k}(y)\}\subseteq B$. As  $B$ being fully invariant we have $f^{k+1}(f^{k}(x))\in B$
	hence $g^{k+1}(f^{k}(x))\in B$. In the same vein, 	$f^{k+1}(f^{k}(y))\in B$ and so $g^{k+1}(f^{k}(y))\in B$. Now, by
	assumption that $(x, y)\in \mathcal{K}_{f^{2k+1}}$, we conclude that $f^{k+1}(f^{k}(x))=f^{k+1}(f^{k}(y))$ therefore
	$g^{k+1}(f^{k}(x))=g^{k+1}(f^{k}(y))$ or $(f^{k}(x),f^{k}(y))\in\mathcal{K}_{g^{k+1}}=\mathcal{K}_{g^{k}}$. Thus, $g^{k}(f^{k}(x))=g^{k}(f^{k}(y)),$  as $g=f|_B$ we	have $f^{k}(f^{k}(x))=f^{k}(f^{k}(y))$ i.e $(x,y)\in\mathcal{K}_{f^{2k}}$.
	In all cases we have proved that $\mathcal{K}_{f^{2k+1}}\subseteq\mathcal{K}_{f^{2k}}$.
	So $A$ is strongly Hopfian.
\end{proof}
%%%%%%%%%%%%%%%%%%%%%%%%%%%%%%%%%%%%%%%%%%%%%%%%%%%%%%%%%%%%%%%%%%%%%%%%%%%%%%%%%%%%%%%%%%%%%%%%%%%
\section{Strongly Hopfian and strongly co-Hopfian quasi-injective and quasi	projective $S$-acts}
%%%%%%%%%%%%%%%%%%%%%%%%%%%%%%%%%%%%%%%%%%%%%%%%%%%%%%%%%%%%%%%%%%%%%%%%%%%%%%%%%%%%%%%%%%%%%%%%%%%%%%%
Hopfian objects and co-Hopfian objects have an elementary interaction with projective objects and injective objects. 
In fact the two results are:
\begin{proposition}
	An injective Hopfian $S$-act is co-Hopfian. Also, a projective 	co-Hopfian $S$-act is Hopfian.
\end{proposition}
\begin{proof}
	Let $A$ be an injective Hopfian $S$-act and $f: A\to A$ be an injective endomorphism. By injectivity of $A$, 
	$$\xymatrix{A\ar[d]_{\id_A}\ar[r]^{f} &A\ar[dl]^{g}\\ A &}$$
	$f$ factor through the identity map $\id_A$ on $A$, yielding an $S$-map $g: A\rightarrow A$ such that $gf=\id_A$. As a
	result, $g$ is a surjective endomorphism on Hopfian $S$-act $A$. So	is an automorphism. Therefore, $f$ is necessarily the inverse
	automorphism to $g$. So that $A$ is co-Hopfian. The proof can be dualized to prove the second	statement.
\end{proof}
In fact, the above results are true for any category. But, for strongly Hopfian and strongly co-Hopfian $S$-acts, we consider
quasi-injective and quasi-projective $S$-acts and deduce the following results.
\begin{definition}
	{\rm An $S$-act $A$ is called {\it quasi-injective} if for any injective $S$-act map $g$ from $B$ to $A$, and any  $S$-act map
		$f$ from $B$ to $A$, there exists an endomorphism $h$ of $A$ such that $f=hg,$ i.e., the following diagram
		$$\xymatrix{B\ar[d]_{f}\ar[r]^{g} &A\ar[dl]^{h}\\ A &}$$
		commutes.}
\end{definition}
\begin{definition}
	{\rm Let $A$ be an $S$-act. The monoid $(\en_S(A), \circ )$ is called 	{\it strongly $\pi$-regular} if for every $f\in \en_S(A)$ there exists $g\in \en_S(A)$ and an integer $n$ such that $f^{n}=gf^{n+1}=f^{n+1}g$.}
\end{definition}
\begin{example}
	{\rm Consider $\mathbb{Z}$-act $\mathbb{Q}.$ We have already seen that	its endomorphisms	are zero or automorphisms (see Example \ref{Q}). Then $\mathbb{Q}$ is strongly $\pi$-regular since for the first case we take the $S$-act map $g=0$ and for the second case we take the $S$-act map $g$ to be the inverse of $f$.}
\end{example}
\begin{proposition}\label{strongly regular}
	Let $A$ be an $S$-act. If $\en_S(A)$ is strongly $\pi$-regular then $A$ is both strongly Hopfian and strongly co-Hopfian.
\end{proposition}
\begin{proof}
	For every $f\in\en_S(A)$ there exists $g\in\en_S(A)$ and an integer $n$ such that $f^n=gf^{n+1}=f^{n+1}g$ and so
	$\mathcal{K}_{f^n}=\mathcal{K}_{gf^{n+1}}.$ Also, we have
	$\mathcal{K}_{f^{n+1}}\subseteq\mathcal{K}_{gf^{n+1}}$ hence $\mathcal{K}_{f^{n+1}}\subseteq\mathcal{K}_{f^n}$.
	Therefore $A$ is strongly Hopfian (see Proposition~\ref{strongly Hopfian}).
	Again by the assumption we have $\im{f^n}=\im{f^{n+1}g}\subseteq\im{f^{n+1}}$, this means $\im{f^n}=\im{f^{n+1}}$, hence $A$ is strongly co-Hopfian (see Proposition \ref{strongly co-Hopfian}).
\end{proof}
\begin{theorem}\label{quasi-in+st Hopf}
	If $A$ is a quasi-injective strongly Hopfian $S$-act for which the endomorphisms monoid is commutative, then $A$ is strongly co-Hopfian.
\end{theorem}
\begin{proof}
	If $f\in \en_S(A)$, then by Proposition~\ref{strongly Hopfian} there exists an
	integer $n$ such that $\mathcal{K}_{f^n}\bigcap \mathcal{I}_{f^n}=\Delta_A.$ We
	define $g: \im {f^n}\rightarrow A$ by $g({f^n}(x)):=f^{n+1}(x)$ for every
	$x\in A$. If $(f^{n}(x),f^{n}(y))\in\mathcal{K}_g$ then we have
	$g(f^{n}(x))=g(f^{n}(y))$ hence $f^{n+1}(x)=f^{n+1}(y)$ i.e,
	$(x,y)\in \mathcal{K}_{f^{n+1}}$. Again  applying Proposition~\ref{strongly Hopfian}, we have
	$(x,y)\in\mathcal{K}_{f^n}$ i.e.,  $f^{n}(x)=f^{n}(y)$. Therefore
	$\mathcal{K}_g=\Delta_{\im f^n}$ and so $g$ is injective. Let $\gamma: \im {f^n}\rightarrow A$ be the  inclusion, defined by
	$\gamma(f^{n}(x)):=f^{n}(x).$ Then by quasi-injectivity of $A$,
	there exists an endomorphism $h : A\rightarrow A$ such that $\gamma=hg$, 
	$$\xymatrix{\im f^n\ar[d]_{\gamma}\ar[r]^{g} &A\ar[dl]^{h}\\ A &}$$
	in other words, $f^{n}=h{f^{n+1}}$. Hence $\en_S(A)$ 	is strongly $\pi$-regular. Now, by the previous proposition, $A$ is strongly co-Hopfian.
\end{proof}
\begin{definition}
	{\rm An $S$-act $A$ is called quasi-projective if for any surjective homomorphism $g : A\rightarrow B$ and any homomorphism $f :
		A\rightarrow B$, there exists an endomorphism $h$ of $A$ such that 	$f=gh$ (i.e; there exists $h: A\rightarrow A$ such that the
		following diagram 	$$\xymatrix{& A \ar[d]^{f} \ar[dl]_{h}\\ A\ar@{->>}[r]_{g} & {\rm B}}$$
		commutes.)}
\end{definition}
\begin{theorem}\label{quasi-proj+st co-Hopf}
	If $A$ is a quasi-projective strongly co-Hopfian $S$-act for which the endomorphisms monoid is commutative, then $A$ is strongly Hopfian.
\end{theorem}
\begin{proof}
	Given $f\in\en_S(A)$, then by Proposition~\ref{strongly co-Hopfian}, there exists an integer $n$ such that
	$\mathcal{I}_{f^n}=\mathcal{I}_{f^{n+1}}$ and so $\im {f^n}=\im{f^{n+1}}$. We define $g:A\rightarrow \im {f^n}$ by $g(a):=f^{n+1}(a)$ for every $a\in A$,	which is a surjective homomorphism since $\im{f^n}=\im{f^{n+1}}.$ For every $x\in \im {f^n}$ there exists $u\in A$ such that $g(u)=f^{n+1}(u)=x$. Now by quasi-projectivity of $A$, there exists an
	endomorphism $h : A\rightarrow A$ such that the following diagram
	$$\xymatrix{& A \ar[d]^{f^{n}} \ar[dl]_{h}\\ A\ar@{->>}[r]_{g} & {\rm im }f^n}$$
	commutes. In other words, $f^{n+1}h=f^n$. Hence $\en_S(A)$ is strongly
	$\pi$-regular. By  Proposition~\ref{strongly regular}, $A$ is strongly Hopfian.
\end{proof}
%%%%%%%%%%%%%%%%%%%%%%%%%%%%%%%%%%%%%%%%%%%%%%%%%%%%%%%%%%%%%%%%%%%%
\section{Finitely generated strongly Hopfian or strongly co-Hopfian $S$-acts}
%%%%%%%%%%%%%%%%%%%%%%%%%%%%%%%%%%%%%%%%%%%%%%%%%%%%%%%%%%%%%%%%%%%%%%%%%%
Finitely generated acts is of as great importance in the theory of acts as it is in the theory of modules. In analogy to modules category, in this short section we will give two results regarding this class of acts, but let us have a definition first.
\begin{definition}
	{\rm An act $A_S$ is called {\it Fitting}  if it  is both strongly Hopfian and strongly co-Hopfian act.}
\end{definition}
\begin{proposition} 
	Let $A_S$ be a finitely generated $S$-act. Then the following statements are equivalent:\\
	$(1)$ Every factor act $A_S/\rho$ of $A$ is co-Hopfian for every congruence  $\rho$ on  $A$.\\
	$(2)$ Every factor act $A_S/\rho$ of $A$ is strongly co-Hopfian for every congruence  $\rho$ on  $A$. 
\end{proposition}
\begin{proof}
	(1)$\Rightarrow$(2) Suppose that $\{ a_1, \cdots , a_n\}$ be a generating set for $A$, i.e., $A_S =\bigcup_{i=1}^{n}a_iS$.  It is sufficient  to show that $A_S$ is  strongly co-Hopfian when $A_S$ is   co-Hopfian. Let   $f\in\en_{S}(A)$ and put $\rho =\bigcup_{k\geq 1}\mathcal{K}_{f^k}$. Then it is easy to see that $\rho$ is a congruence on $A$.  Define the induced map $g$ in the following diagram 
	$$
	\xymatrix{A_S\ar[r]^{\pi}\ar[d]_{f} &A_S/ \rho\ar[d]^{g}\\A_S\ar[r]_{\pi}&A_S/\rho }
	$$ 
	as $g([a]_\rho )= [f(a)]_\rho$. It is well-defined because if  $[a]_\rho = [b]_\rho$ then there exists $k\geq 1$ such that $(a,b)\in \mathcal{K}_{f^k}$. That is $f^k(a) = f^k(b)$. Since $f^{k+1}(a) = f^{k+1}(b)$ by the definition of $\rho$, $(f(a) , f(b))\in \rho$ so $[f(a)]_\rho= [f(b)]_\rho$. Next, it is clear that $g$ is an  $S$-homomorphism. Moreover, $g$ is injective, because $(a,b)\in \rho$ if and only if $(f(a) , f(b))\in \rho$. By co-Hopficity of $A_S/\rho$, $g$ is surjective. Then there exist  $ b_1, \cdots , b_n\in A_S$ such that $g([b_i]_\rho )= [a_i]_\rho$, which is equivalent to $(f(b_i), a_i)\in \rho$. For any $1\leq i\leq n$ there exists $m_i\geq 1$ such that  $(f(b_i), a_i)\in \mathcal{K}_{f^{m_i}}.$ Now take $m=\max\{m_1, \cdots , m_n\}.$ As $f$ is a homomorphism it is easily seen that $\im f^{m+1} =\im f^{m}$. Therefore, $\ci_{f^{m+1}} =\ci_{f^m}$ and hence $A_S$ is strongly co-Hopfian. \\
	(2)$\Rightarrow$(1) This is clear.
\end{proof}
\begin{proposition}
	Let $A_S$ be a finitely generated $S$-act. Then the following statements are equivalent:\\
	$(1)$ Every factor act 	$A_S/\rho$  of $A$ is Hopfian and co-Hopfian.\\
	$(2)$ Every factor act 	$A_S/\rho$ of $A$ is a Fitting act.
\end{proposition}
\begin{proof}
	(1)$\Rightarrow$(2)
	As in the proof of previous proposition it is enough to see that $A_S$ is a Fitting act. We know by the earlier proposition that $A_S$ is strongly co-Hopfian. Let $f$ be an endomorphism of $A_S$, then 	there exists an integer $m$ such that $\mathcal{I}_{f^m}=\mathcal{I}_{f^{m+1}}$. Consider now $h$ the endomorphism of $\im f^m$ defined by $h(y) = f^m(y),$ this endomorphism is surjective. By Homomorphism Theorem for acts~\cite[Theorem I.4.21]{KKM}, $\im f^m$ is isomorphic to the factor act $A_S/\mathcal{K}_{f^m}$, and in light of the hypothesis it is Hopfian. Therefore, $h$  is injective. This forces $\mathcal{I}_{f^m}\cap \mathcal{K}_{f^m} = \bigtriangleup_{A}$, that is, $A_S$ is strongly Hopfian by Proposition~\ref{strongly Hopfian}, and hence  $A_S$  is a Fitting act.\\
	(2)$\Rightarrow$(1) This is clear.
\end{proof}
%\vspace{4mm}\noindent{\bf Acknowledgements}\\
%We would like to thank the referee for  reading the paper carefully and giving helpful and constructive  comments.
%%%%%%%%%%%%%%%%%%%%%%%%%%%%%%%%%%%%%%%%%%%%%%%%%%%%%%%%%%%%%%%%%%%%%%%%%%%%%%%%%%%%%%%%%%%%%%%%%%%%%%%%

%\noindent \textbf{Addresses of authors:\medskip }
\end{document}